\documentclass[a4paper,11pt]{amsart}

\usepackage{xypic}
\usepackage{subfigure}
\usepackage{graphicx}
\usepackage{graphics}

\usepackage{xspace,enumerate}
\usepackage[letterpaper,top=2.8cm, bottom=2.8cm, left=3.3cm, right=3.3cm]{geometry}

\newtheorem{theorem}{Theorem}[section]
\newtheorem{lemma}[theorem]{Lemma}
\newtheorem{proposition}{Proposition}[section]

\theoremstyle{definition}
\newtheorem{definition}[theorem]{Definition}
\newtheorem{example}[theorem]{Example}

\theoremstyle{remark}

\newcommand{\NW}{\operatorname{NW}}

\numberwithin{equation}{section}

\begin{document}

\title[Proof of the Pruning Front Conjecture]{Proof of the Pruning Front Conjecture for 
certain H\'enon parameters}

\author{Valent\'in Mendoza}
\address{Instituto de Matem\'atica e Estat\'istica, 
 Universidade de S\~ao Paulo. Rua do Mat\~ao 1010, Cidade Universitaria, CEP 05508-090, Brazil.}
\curraddr{}
\email{jvalentm@ime.usp.br}
\thanks{Research supported by FAPESP grant 2007/55771-0.}

\subjclass[2010]{Primary }


\date{Dec 3, 2011.}

\dedicatory{}

\begin{abstract}
The {\em Pruning Front Conjecture} is proved for an open set of H\'enon parameters 
far from unimodal. More specifically, for an open subset of H\'enon parameter space,
 consisting of two connected components one of which intersects the area-preserving locus,
 it is shown that the associated H\'enon maps are prunings of the horseshoe. In particular, 
their dynamics is a subshift of the two-sided two-shift. 
\end{abstract} 

\maketitle

\section{Introduction}

Pruning was introduced by Cvitanovi\'c \cite{Cvi1} to describe the dynamics of the H\'enon (or Lozi) family. 
The main idea of pruning is to describe  H\'enon maps as {\em partially formed} horseshoes or, put another 
way, as horseshoes from which part of the dynamics was destroyed or {\em pruned away}. 
Roughly, the {\em Pruning Front Conjecture (PFC)} states just that: every map in the H\'enon 
family is a pruning of the horseshoe, which is proved here for an open set of 
parameter values far from the degenerate unimodal world. More precisely, it is shown that 
there are two open topological disks in parameter space, one of which intersects the 
area-preserving locus (and contains parameters previously studied by Davis, MacKay and 
Sannami \cite{DavMacSan1}), whose associated H\'enon maps are non-trivial prunings of the horseshoe.    

\smallskip

Pruning can be described as a way  to give a topological description of the dynamics of 
surface homeomorphisms and the mechanism of creation or destruction of chaos in 2-dimensions. 
A means to formalize this idea, due to de Carvalho\footnote{A different way to formalize 
pruning for the Lozi family was introduced by Ishii in \cite{Ishii} which enabled him to
 prove the PFC for that family.}, is to introduce {\em pruning isotopies} which destroy 
pieces of the dynamics of a surface homeomorphism $f_0$ in a controlled way \cite{dCar1}.
Here \textit{to destroy dynamics} means to convert non-wandering points into 
wandering points, and \textit{in a controlled way} means that this destruction occurs
in regions which we can define explicitly in terms of $f_0$. A precise formulation of this is 
as follows:
\begin{definition}\label{def:isotopy}
 Let $D$ be a domain of the plane, $f_0\colon\mathbb{R}^2\rightarrow\mathbb{R}^2$ a homeomorphism
of the plane. Let $\mathbb{P}:=\cup_{i\in\mathbb{Z}}f_0^i(D)$ be the 
saturation of $D$ under $f_0$. An isotopy $f_t\colon f_0\simeq f_1$ is 
a \textit{pruning isotopy} if:
\begin{enumerate}
\item $\operatorname{Supp}(f_t)\subset\mathbb{P}$, and
\item $\NW(f_1)=\NW(f_0)\setminus\mathbb{P}$,
\end{enumerate}
where $\NW(f)$ denotes the non-wandering set of $f$.
\end{definition}
In \cite{dCar1} conditions on $D$ were introduced to ensure the existence of pruning 
isotopies; a disk $D$ satisfying  these conditions is called a \textit{pruning disk}.
\begin{theorem}[Pruning theorem]\label{thm:poda}
 Given a pruning disk $D$ for a homeomorphism $f_0$, there exists a pruning isotopy
associated to $D$.
\end{theorem}
In words, a pruning disk $D$  for a homeomorphism  $f_0$ is a topological disk
 for which  it is possible to destroy all of the orbits of $f_0$ which enter $D$, while
 leaving other  orbits untouched.

 The homeomorphism $f_1$ is called \textit{the pruning homeomorphism of $f_0$ associated
 to the disk $D$}. It is important to see that pruning theory gives models to understand
 the topological dynamics of homeomorphisms. If we start with a homeomorphism $f_0$ we 
can construct the \textit{pruning family of $f_0$, $\mathcal{P}(f_0)$}, defined to be 
the closure of the  family of all the homeomorphisms obtained applying the pruning 
theorem to $f_0$ a finite number of times. Depending on $f_0$, the pruning family 
$\mathcal{P}(f_0)$ contains infinitely many different models of dynamics. For example, 
if $f_0=F$ is the Smale horseshoe it can be proved that $\mathcal{P}(F)$ has 
uncountably many topological models. See \cite{dCar1} for these facts.

In this paper, we use pruning to study the H\'enon family \cite{Hen1} which is the
 family of diffeomorphisms of the plane defined by
\begin{equation*}
 H_{a,b}(x,y)=(a-x^2-by,x), \textrm{  } a,b\in\mathbb{R},b>0.
\end{equation*} 
Our work is related to the \textit{Pruning Front Conjecture (PFC)}; which states 
that each H\'enon map can be understood as a partially formed horseshoe, i.e., what
 is obtained from the horseshoe after pruning some orbits:

\textbf{Pruning Front Conjecture (PFC)}.  \textsl{Up to semiconjugacy, the real H\'enon family
is contained in $\mathcal{P}(F)$, that is, for any choice of  real parameter values $a,b$,
the H\'enon map $H_{a,b}$ belongs to $\mathcal{P}(F)$, up semiconjugacy.} 

Many authors have found numerical evidence that for certain choices of parameters $(a,b)$,
 the non-wandering set of $H_{a,b}$ is a subset of the  Smale horseshoe, e.g. \cite{CviGunPro1,DavMacSan1}.
 In particular, Davis, Mackay and Sannami \cite{DavMacSan1} give parameter values $(a_1,b_1)=(5.4,1)$ 
 where the non-wandering set of $H_{a,b}$ seems to be a subshift of finite type. On other hand,
in \cite {Arai1} Arai showed a rigorous computational method to prove hyperbolicity
for certain parameter values in the H\'enon family. These parameter values contain
$(a_1,b_1)$ above and $(a_2,b_2)=(2.25,0.25)$. 

We use the above facts to prove the PFC in some open neighborhoods around $(a_1,b_1)$ and $(a_2,b_2)$. 
Our main theorem is the following.
\begin{theorem}\label{theo:conjecture}
 There exists an open set $A=A_1\cup A_2$, where $A_1$ and $A_2$ are open topological disks 
in the real parameter plane, such that if $(a,b)\in A$ then $H_{a,b}$ is topologically 
conjugate to a pruning homeomorphism of the horseshoe.
\end{theorem}

In $\S$\ref{sectionpruningdisk}  we will define  which pruning disks we are going to use. 
In $\S$\ref{pruninghenon}, we present some concepts and results for the complex H\'enon 
family and using Arai's theorem \cite{Arai2}, we show Theorem \ref{theo:conjecture}.

\section{The horseshoe, Pruning theory and shift automorphisms}\label{sectionpruningdisk}
\subsection{The Smale Horseshoe}\label{subs:homoclinic}
In this section we recall the definition of the Smale horseshoe  and set the notation 
to be used. The Smale horseshoe $F$ is a hyperbolic diffeomorphism of the plane whose dynamics 
in its non-wandering  set is conjugate to the shift $\sigma$ in 
$\Sigma_2:=\{\texttt{0},\texttt{1}\}^\mathbb{Z}$. 

Consider the square $Q=[-1,1]\times[-1,1]$ joined to two semi-disk of radii $1/2$  
and centered in $(-1,0)$ and $(1,0)$.  Foliate the square $Q$ with horizontal unstable 
leaves and vertical stable leaves, and begin by choosing the action of $F$ on $Q$ as 
depicted in Figure \ref{figurahorseshoe}. We require 
that $F$ should stretch the unstable leaves uniformly and contract the stable leaves uniformly by 
factors $\lambda$ and $\mu$, respectively, with $\lambda>2$ and $\mu<1/2$.

\begin{figure}
\centering
\includegraphics[width=80mm,height=60mm]{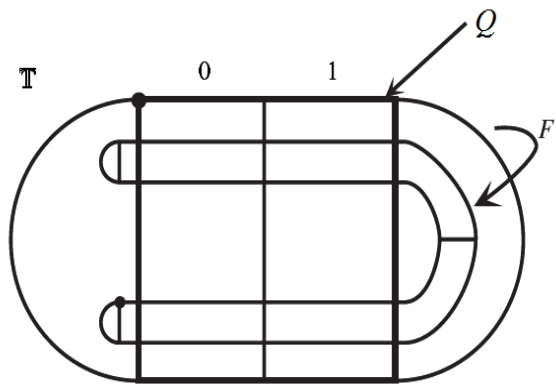}
\caption{Dinamics of $F$.}
\label{figurahorseshoe}
\end{figure}

Then it is possible to prove that  the non-wandering set $\NW(F)$ is a Cantor set in the plane
together with a fixed point $x$ in the left semi-circle. Making a homotopy we can suppose that
 $x$ is the point $\infty$ so the non wandering set is only the Cantor set.
We also suppose that if $z\in\NW(F)$ then $W^s(z)\cap Q$ is a vertical segment 
and $W^u(z)\cap Q$ is a horizontal segment.

 The conjugacy $\vartheta:\NW(F)\rightarrow\Sigma_2$ between the horseshoe $F$ and
the shift $\sigma$ is chosen to be usual one, that is, for any $z=(x,y)$ in $\NW(F)$ 
the symbol sequence associated is $\vartheta(z)=(s_i)_{i\in\mathbb{Z}}$ where:
\begin{equation}
  s_i:=\left\{ \begin{array}{cc}
 \texttt{0} & \textrm{ if $\pi_1(F^{i}(z))<0$}\\
\texttt{1} & \textrm{ if $\pi_1(F^{i}(z))>0$}
\end{array}
\right.
\end{equation}
and $\pi_1$ is the projection on the first coordinate $\pi_1(z):=x$. In the following, we 
identify one point $z\in\NW(F)$ by its \textit{code} or symbolic representation $h(z)$ in $\Sigma_2$.
 When writing elements $s\in\Sigma_2$, it is common to juxtapose a point between $s_{-1}$
 and $s_0$ to indicate the origin of the sequence. 

\subsection{Pruning}
We now present the main definitions and results from \cite{dCar1}. We adapt 
them to the needs of this paper and avoid technical details will not be used here. 

In \cite[Section 1]{dCar1} a pruning disk is defined to be an open topological disk $D$, 
whose closure is a closed topological disk, its boundary can be written as
 $\partial D=C\cup E$, where $C$ and $E$ are arcs which meet only at their endpoints, and
 satisfy the dynamical conditions below plus some technical conditions. 

A  disk $D$ is \emph{a pruning disk} if 
\begin{equation}\label{equcond1}
 F^n(C)\cap D=\emptyset \textrm{   and  } F^{-n}(E)\cap D=\emptyset, \forall n\ge1
\end{equation}
and 
\begin{equation}\label{equcond2}
\lim_{n\rightarrow\infty}\operatorname{diam}(F^n(C))=0 \textrm{ and } \lim_{n\rightarrow\infty}\operatorname{diam}(F^{-n}(E))=0.
\end{equation}

For our purposes, $C$ and $E$ will always be segments of stable, unstable manifolds, respectively, and 
the definition is made using the concept of homoclinic pruning disk below.
\begin{definition}\label{definitiondisk}
A \emph{homoclinic pruning disk} is a open topological disk $D\subset\mathbb{R}^2$ if 
$D$ satisfies (\ref{equcond1}) and (\ref{equcond2}) and $C$ and $E$ satisfy:
\begin{enumerate}
\item[(i)] $E\subset W^u(\texttt{0}^\infty)$ and $C \subset W^s(\texttt{0}^\infty)$,
\item[(ii)] there is an open set $G$ with $\overline{G}\subset D$ and $D\cap\NW(F)\subset G$.
\end{enumerate}
\end{definition}
Observe that the conditions ensure that the stable and unstable manifolds of all points in $\NW(F)$ do not 
accumulate internally on $\partial D$. 
\begin{example}
Let $D$ be the homoclinic disk defined by the homoclinic points 
$p_0=\texttt{0}^\infty\texttt{1111}\cdot\texttt{10110}^\infty$ and
 $p_1=\texttt{0}^\infty\texttt{1110}\cdot\texttt{10110}^\infty$. See Figure \ref{homoclinicdisk}.
 In this case, we need three forward iterates of $p_0$ and $p_1$ to see that, $\forall n>0$,
the segment $F^n(C)$ does not intersect $D$, and three backward iterates to see
that $F^{-n}(E)$ does not intersects $D$.
\begin{figure}
\centering{\includegraphics[width=93mm,height=83mm]{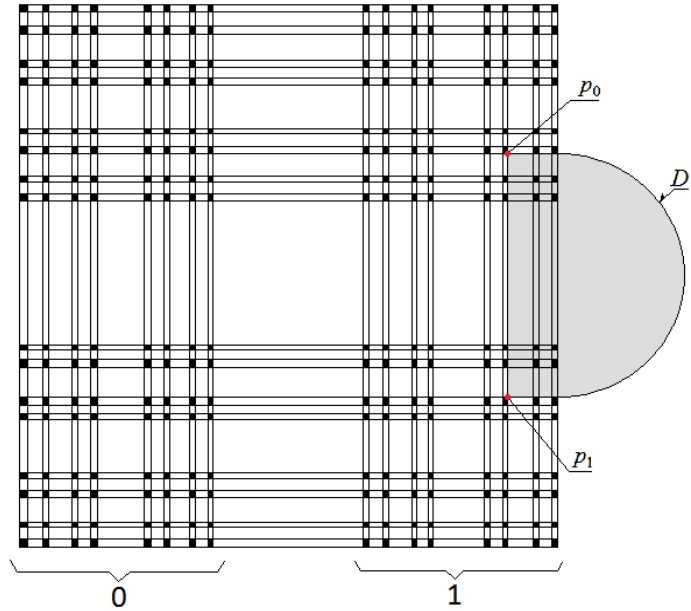}}
\caption{\small{Homoclinic disk.}}\label{homoclinicdisk}
\end{figure}
\end{example}

 We define  \emph{pruning region associated to $D$} to be the open set 
$\mathbb{P}:=\cup_{i=-\infty}^{\infty}F^{i}(D)$. The set
 $\mathbb{P}_s:=\vartheta(\mathbb{P}\cap\NW(F))$ is called the \textit{symbolic pruning region}.

Now let $f_0:=F$ be the horseshoe. Then it follows from \cite{dCar1} that:
\begin{theorem}[Pruning Theorem]\label{theopruning}
 If $D$ is a homoclinic pruning disk, then there exists a pruning isotopy 
$f_t\colon f_0\simeq f_1$ of $f_0$.
 \end{theorem}
It follows that the dynamics of the pruning map $f_1$ is what remains after removing
the pruning region. Symbolically this is:
\begin{lemma}\label{lemmapruning}
The restriction $\vartheta_s=\vartheta|_{\NW(f_1)}$ is a conjugacy between
 $f_1|_{\NW(f_1)}$ and $\sigma|_{\Sigma_2\setminus\mathbb{P}_s}$.
\end{lemma}
Because we are considering only homoclinic disks as in Definition \ref{definitiondisk}, 
the subshifts obtained by this procedure will always be of finite type. More general 
subshifts can be obtained considering general pruning disks but they will not occur here.

The \textit{pruning family} $\mathcal{P}(F)$ is defined to be the closure, 
 in the $C^0$-topology, of the set of  all the prunings of $F$. It can be shown
 that $\mathcal{P}(F)$ contains uncountably many different models of dynamics 
and the Pruning Front Conjecture states that $\mathcal{P}$ contains enough 
models to describe \textit{all} H\'enon maps.

\subsection{Pruning automorphisms}
In this section certain automorphisms of the shift will be discussed.
They will be called pruning automorphisms.

Given $N\in\mathbb{N}$ and $s\in\{0,1\}$ we set $s^N:=\underbrace{s_is_i...s_i}_{N-times}$ 
and $s^0:=\emptyset$ as the empty word. 

Define the homoclinic disk $D_{N,M}$ of the horseshoe $F$ to be the disk that contains the blocks
 corresponding to the symbolic sequences $\texttt{0}^N\texttt{1}\cdot\texttt{01}\texttt{0}^M$ 
and $\texttt{0}^N\texttt{1}\cdot\texttt{11}\texttt{0}^M$. See Figure \ref{figurediskNM}. Then the disk
$D_{N,M}$ is bounded by the stable and unstable segments passing through the homoclinic 
points 
$p_0=\texttt{0}^\infty\texttt{11}\texttt{0}^{N-1}\texttt{1}\cdot\texttt{11}\texttt{0}^{M-1}\texttt{11}\texttt{0}^\infty$ 
and
 $p_1=\texttt{0}^\infty\texttt{11}\texttt{0}^{N-1}\texttt{1}\cdot\texttt{01}\texttt{0}^{M-1}\texttt{11}\texttt{0}^\infty$.

\begin{figure}[h]
\centering{\includegraphics[width=70mm,height=75mm]{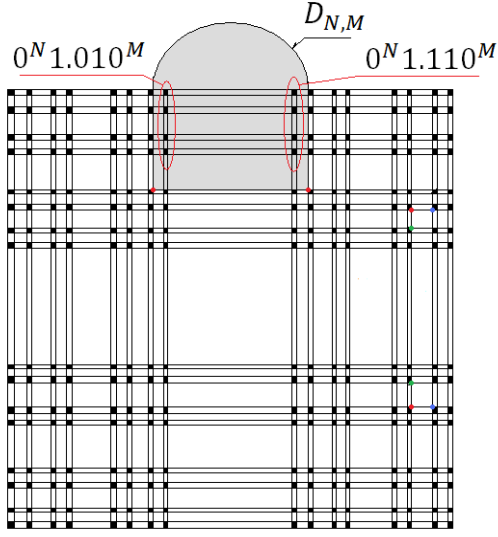}}
\caption{\small{Pruning disk $D_{N,M}$ with $N=1$ and $M=2$.}}\label{figurediskNM}
\end{figure}

\begin{proposition}\label{prophpruningdisk}
 $D_{N,M}$ is a pruning disk, except when $N=0$ and $M=0$; $M=0$ and $N=1$; 
$N=1$ and $M=1$; or, $M=1$ and $N=0$.
\end{proposition}
The proof consists of following the orbits of $C$ and $E$ just as in figure \ref{homoclinicdisk}.
Note that $F(D_{N,M})$ is positioned similarly to the pruning disk of that figure.
\bigskip

Let $f_{N,M}$ be the pruning homeomorphism associated to $D_{N,M}$ as in theorem \ref{theopruning}. 

An automorphism of the shift is a shift commuting homeomorphism from $\Sigma_2$ to $\Sigma_2$.
The set of all the automorphisms of the shift is denoted by $\operatorname{Aut}(\Sigma_2)$, 
and $\operatorname{Fix}(\rho)$ denote the set of fixed points of a automorphism $\rho$.

The definition below associates an automorphism of the shift to a pruning disk $D_{N,M}$.
These are called \textit{pruning automorphisms}.
\begin{definition}\label{def:auto}
 Let $\rho_{N,M}\in\operatorname{Aut}(\Sigma_2)$ be the shift automorphism that interchanges
 the sequences $\texttt{0}^N\texttt{1010}^M$  and $\texttt{0}^N\texttt{1110}^M$:
\begin{equation}
 (\rho_{N,M}(\emph{s}))_i=\left\{\begin{array}{ccc}
\texttt{0} & \textrm{if $s_{i-N-1}...s_i...s_{i+M+1}=\texttt{0}^N\texttt{1110}^M$}\\
\texttt{1} & \textrm{if $s_{i-N-1}...s_i...s_{i+M+1}=\texttt{0}^N\texttt{1010}^M$}\\
s_i & \textrm{otherwise}
\end{array}
\right.
\end{equation}
\end{definition}
In figure \ref{figurediskNM} this exchangies points in $D$ to the left and right of the vertical
center line, among many possible other symmetries. 
The automorphism $\rho_{N,M}$ is called \textit{pruning automorphism associated to the disk $D_{N,M}$}.
 
By the above definition, $\Sigma_2\setminus\mathbb{P}_s=\operatorname{Fix}(\rho_{N,M})$ where 
$\mathbb{P}_s$ is the symbolic pruning front associated to $D_{N,M}$. Hence, by lemma \ref{lemmapruning},
$\vartheta_s\colon\NW(f_{N,M})\rightarrow\operatorname{Fix}(\rho_{N,M})$ is a conjugacy  between 
$f_{N,M}$ restricted to $\NW(F)\setminus\mathbb{P}$ and the shift $\sigma$ restricted to 
$\operatorname{Fix}(\rho_{N,M})$. Therefore, in $\NW(f_{N,M})$, we have the following equation:
\begin{equation}\label{equ:conjugacy}
 \vartheta_s\circ f_{N,M}=\sigma\circ\vartheta_s.
\end{equation}

\section{The H\'enon family}\label{pruninghenon}
\subsection{The horseshoe locus}
The H\'enon family has an extension, also denoted $H_{a,b}$, to $\mathbb{C}^2$, 
with $(a,b)\in\mathbb{C}\times\mathbb{C}^*$. This has been investigated
 by many authors. See, e.g., \cite{BedSmi1,BedSmi2,HubObe1,HubObe2}.
Here we state some definitions and results of real and complex H\'enon maps.

Define
\begin{equation}
 K_{a,b}^\mathbb{C}=\{p\in\mathbb{C}^2|\{H_{a,b}^n(p)\}_{n\in\mathbb{Z}}\textrm{ is bounded }\}
\end{equation}
to be the set of bounded orbits and $K_{a,b}^{\mathbb{R}}:=K_{a,b}^{\mathbb{C}}\cap\mathbb{R}^2$.
 Let $\mathcal{H}^\mathbb{C}$ be the set of $(a,b)\in\mathbb{C}\times\mathbb{C}^*$ such that the 
restriction of $H_{a,b}$ to $K_{a,b}^{\mathbb{C}}$ is topologically conjugate to the full
 2-shift $(\sigma,\Sigma_2)$.
The set $\mathcal{H}^\mathbb{C}$ is called the \textit{horseshoe locus}.
Likewise, let $\mathcal{H}^{\mathbb{R}}$ be the set of
$(a,b)\in\mathbb{R}^2$ such that the restriction of $H_{a,b}$ to $K_{a,b}^{\mathbb{R}}$ is 
topologically conjugate to the full 2-shift. 

Devaney and Nitecki showed that $\mathcal{H}^\mathbb{R}\neq\emptyset$ and contains the set
 \begin{equation}
  \operatorname{DN}:=\{(a,b)\in\mathbb{R}^2|a>(5+2\sqrt{5})(1+|b|)^2/4, b\neq0\};
 \end{equation}
 they also prove that if $(a,b)\in\operatorname{EMP}:=\{(a,c)\in\mathbb{R}^2: a<-(1+|b|)^2/4\}$ 
then $K_{a,b}^{\mathbb{R}}=\emptyset$ (See \cite{DevNit1}).

Hubbard and Oberste-Vorth (\cite{Obe1}) showed that $\mathcal{H}^\mathbb{C}$ contains the set 
\begin{equation}
 \operatorname{HOV}:=\{(a,b)\in\mathbb{C}^2:|a|>2(1+|b|)^2,b\neq0\}
\end{equation}
Observe that $\operatorname{DN}\subset\operatorname{HOV}$.

In \cite{Arai2}, Arai shows that there are parameters values in $(a,b)\in\mathcal{H}^\mathbb{C}\cap\mathbb{R}^2$
such that $K^{\mathbb{R}}\neq\emptyset$ but $H_{a,b}$ restricted to $K^{\mathbb{R}}$ is not 
conjugate to a full horseshoe. Those were called \textit{parameters of type-3}. 
Also, in \cite{BedSmi2}, Bedford and  Smillie exhibited a way to join distinct connected
 components of the \textit{real horseshoe locus}, 
$\mathcal{H}^\mathbb{C}\cap\mathbb{R}^2$, using  loops in $\mathcal{H}^\mathbb{C}$. This
can be accomplished by loops starting in $\operatorname{DN}$.  Let $\mathcal{H}_0^{\mathbb{C}}$ be 
the connected component of $\mathcal{H}^{\mathbb{C}}$ that contains $\operatorname{HOV}$ and take a point 
$(a_0,b_0)\in\operatorname{DN}\subset\mathcal{H}_0^{\mathbb{C}}$. Then there exists a conjugacy 
$h_0:K_{a_0,b_0}^{\mathbb{C}}\rightarrow\Sigma_2$. Let $\gamma(t)$ be a loop in 
$\mathcal{H}_0^{\mathbb{C}}$ with basepoint in $(a_0,b_0)$. By structural stability there
 are conjugacies $h_t$ between $K_{\gamma(t)}^{\mathbb{C}}$ and $\Sigma_2$. 
Since $\gamma(0)=\gamma(1)=(a_0,b_0)$, we can define an automorphism of the 2-shift
\begin{equation}
 \rho(\gamma):=h_1\circ h_0^{-1}\in \operatorname{Aut}(\Sigma_2).
\end{equation}
Thus, $\rho$ sends a loop $\gamma$ to an automorphism in $\operatorname{Aut}(\Sigma_2)$ which only depends on
the homotopy class of the loop,  $[\gamma]$. Therefore, the map:
\begin{equation}
 \rho:\pi_1(\mathcal{H}_0^{\mathbb{C}},(a_0,b_0))\rightarrow\operatorname{Aut}(\Sigma_2)
\end{equation}
given by $[\gamma]\rightarrow \rho(\gamma)$ is a group homomorphism.
The map $\rho$ is called the \textit{mono\-dromy homomorphism}. 

Some results about the horseshoe locus were obtained by Arai in \cite{Arai2}. His main result
 relates to the loops in $\mathcal{H}_0^{\mathbb{C}}$ which pass through real parameter 
values $(a,b)$. 
\begin{theorem}[Arai]\label{teoarai}
 If $(a,b)\in\mathcal{H}_0^{\mathbb{C}}\cap\mathbb{R}^2$, then there exists a loop
 $\gamma:=\overline{\alpha}^{-1}\cdot\alpha$, where $\alpha$ is a path in 
$\mathcal{H}_0^{\mathbb{C}}$ starting in $(a_0,b_0)$ and ending in $(a,b)$ (so that $\gamma(1/2)=(a,b)$),
 such that 
\begin{itemize}
\item[(i)] the image of $K^{\mathbb{R}}_{a,b}$ under $h_{1/2}$ is exactly $\operatorname{Fix}(\rho(\gamma))$, 
 and
\item[(ii)] this identification is dynamical, i.e.
\begin{equation}
\sigma\circ h_{1/2}=h_{1/2}\circ H_{a,b}.
\end{equation}
\end{itemize}
\end{theorem}
Arai's proof uses the fact that a H\'enon map $H_{a,b}$ is conjugate to 
$H_{\overline{a},\overline{b}}$ by the complex conjugation 
$\varphi(x,y):=(\overline{x},\overline{y})$.

It is natural to ask what automorphisms in the theorem above are realized by real 
H\'enon maps. In the next section, we will observe that among those automorphisms are some
pruning automorphisms.

\section{the application}
Now we prove a theorem that allows one to show that a parameter pair $(a,b)$ corresponds to
a pruning homeomorphism. Theorem \ref{theo:conjecture} follows as a direct application using the
 rigorous numerical work of Arai \cite{Arai1,Arai2}.
\begin{definition}
 Let $\gamma$ be a closed curve in $\mathbb{C}^2$ and $\overline{\gamma}:=\varphi(\gamma)$. 
In the following we suppose
 $\gamma(0)\in\operatorname{DN}$. We say that $\gamma$ is \textit{symmetric} if $\overline{\gamma}=\gamma^{-1}$, 
that is, $\overline{\gamma}(t)=\gamma(1-t),\forall t\in[0,1]$.
When  $\gamma$ is symmetric it follows that $\gamma(1/2)\in\mathbb{R}^2$. 
\end{definition}

Equation (\ref{equ:conjugacy}) gives us a connection between a pruning homeomorphism and 
the possible real dynamics  of $H_{a,b}$ when $(a,b)\in\mathcal{H}_0^{\mathbb{C}}\cap\mathbb{R}^2$ 
in the following way:

\begin{theorem}\label{teoconjsime}
 Let $\rho_{N,M}$ be a pruning automorphism of the shift. 
Suppose that there exists a symmetric curve $\gamma\in\mathcal{H}_0^{\mathbb{C}}$ such 
that $\rho(\gamma)=\rho_{N,M}$  and define $(a,b)=\gamma(1/2)$. Then there
 exists a conjugacy between the H\'enon map $H_{a,b}$, restricted to
 $K_{a,b}^{\mathbb{R}}$, and the pruning homeomorphism $f_{N,M}$ restricted to $\NW(f_{N,M})$.
\end{theorem}
\begin{proof}
 Let  $h_t$ be the conjugacies between $H_{a,b}$ and $\sigma$ restricted to 
$K_{\gamma(t)}^{\mathbb{C}}$ and $\Sigma_2$, respectively. Then $\rho_{N,M}=h_1\circ (h_0)^{-1}$
 by definition. Let $\alpha:=\gamma|_{[0,1/2]}$. Then 
$\gamma_{N,M}=\overline{\alpha}^{-1}\cdot\alpha$ and by theorem \ref{teoarai} of Arai, 
$h_{1/2}$ is a conjugacy between $H_{a,b}$, restricted to  $K_{a,b}^{\mathbb{R}}$,  and $\sigma$, 
 restricted to $\operatorname{Fix}(\rho_{N,M})$. By equation (\ref{equ:conjugacy}), the 
desired conjugacy  is $(\vartheta_s)^{-1}\circ h_{1/2}$. 
\end{proof}

Now we prove the Theorem \ref{theo:conjecture}.
 \begin{proof}[Proof of Theorem \ref{theo:conjecture}.]
Let $I_1=[5.3125,5.46875]\times\{1\}$ and take $(a,b)\in I_1$. 
By \cite[Theorem 1.1]{Arai1}, these parameter values correspond to hyperbolic H\'enon maps.
 By structural stability there exists a connected open set $A_1$, with $I_1\subset A_1$, such that all its 
maps are conjugate. By \cite[Theorem 4.]{Arai2}, there exists a symmetric curve
 $\gamma_1\in\mathcal{H}_0^{\mathbb{C}}$ such that $\gamma_1(1/2)=(a,b)$ and
$\rho(\gamma_1)=\rho_{2,2}$. Hence, by Theorem \ref{teoconjsime}, $H_{a,b}$, restricted to 
$K_{a,b}^{\mathbb{R}}$, is topologically conjugate to the pruning homeomorphism $f_{2,2}$
 associated to the pruning disk $D_{2,2}$ and restricted to $\NW(f_{2,2})$.

 Let $I_2=[2.21875,2.296875]\times\{0.25\}$. 
 The open set $A_2$ is obtained considering the values $(a,b)\in I_2$.
 The proof follows the same procedure. In this case $H_{a,b}$, restricted to $K_{a,b}^\mathbb{R}$, 
is topologically conjugate to the pruning homeomorphism $f_{0,2}$ associated to the pruning
disk $D_{0,2}$ and restricted to $\NW(f_{0,2})$.
\end{proof}
\section{Further parameters values}
 The technique above can be used to prove further cases of PFC provided the symmetric paths as in 
Theorem \ref{teoconjsime} are found. There is numerical evidence that many other paremeter regions
fall into this proof scheme but this has not been rigorously verified. 
For instance, some parameter values are the following with their
 respective pruning disks.
\begin{table}[h]\small
\begin{center}
\begin{tabular}{|c|c|}
\hline
Parameter values $(a,b)$& Pruning disks \\
\hline
$(3.5,0.55)$ &  $D_{2,3}$  \\
\hline
$(2.766,0.4)$ & $D_{0,3}$ and $D_{1,2}$  \\
\hline
$(2.887,0.4)$ & $D_{1,3}$ and $D_{2,2}$  \\
\hline
$(2.345,0.19)$ & $D_{0,3}$  \\
\hline
\end{tabular}

\end{center}
\end{table}
 
We conjecture that all these parameters are in the horseshoe locus, because, 
numerically, it seems that there are paths in $\mathcal{H}^\mathbb{C}$ joining these
parameters values with points in $DN$: these paths can be found using the program
 \textbf{SaddleDrop}\footnote{SaddleDrop is a program created by J. Hubbard and K. Papadantonakis
to study the complex H\'enon family. It can be found in http://www.math.cornell.edu/~dynamics/SD/index.html} 
and drawing the complex unstable manifold of the fixed point $\texttt{0}^\infty$. 

As an example, consider the parameter values $(a,b)=(2.812,0.4)$. The pruning disks
 conjectured for this map are $D_{1,2}$ and the homoclinic disk $D'$ defined by the homoclinic orbits 
$p_0=\texttt{0}^\infty\texttt{10111}\cdot\texttt{100110}^\infty$ and
$p_1=\texttt{0}^\infty\texttt{10110}\cdot\texttt{100110}^\infty$ (See Figure \ref{homdiskexample}). 
To justify that  $(2.812,0.4)\in\mathcal{H}^\mathbb{C}$, we take
 $9$ points in the complex parameter space which, at least numerically, seem to belong to
 a path joining $(2.812,0.4)$ and the parameter values $(3.149,0.4)\in DN$. To see that 
this path is entirely contained in $\mathcal{H}^\mathbb{C}$, draw (using SaddleDrop) the
 intersection of the set $K_{a,b}$ of bounded orbits with the complex unstable manifold of $\texttt{0}^\infty$. 
These are shown in Figures \ref{complex1} - \ref{complex9}. 
Observing these figures one see what seems to be a Cantor set varying continuously in $\mathbb{C}^2$.
It suggests that these values are in the horseshoe locus.

\begin{figure}
\centering{\includegraphics[width=93mm,height=83mm]{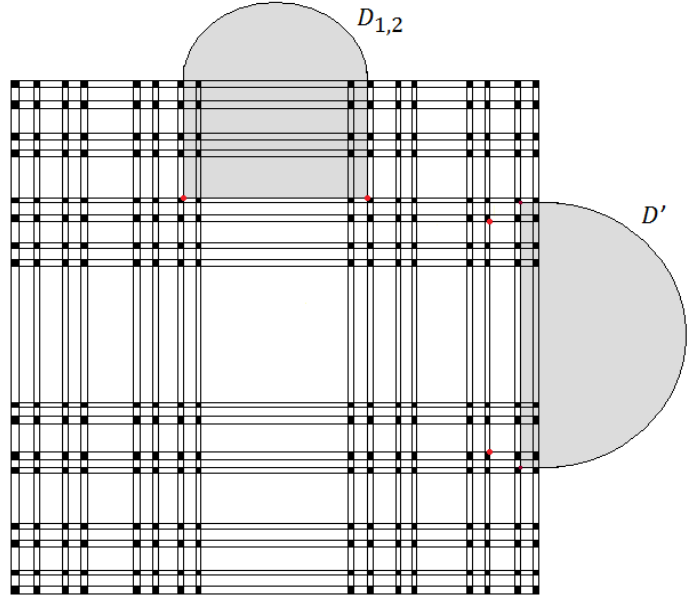}}
\caption{\small{Homoclinic disks conjectured for $H_{2.812,0.4}$}.}\label{homdiskexample}
\end{figure}

\begin{figure}
\centering
\includegraphics[width=90mm,height=40mm]{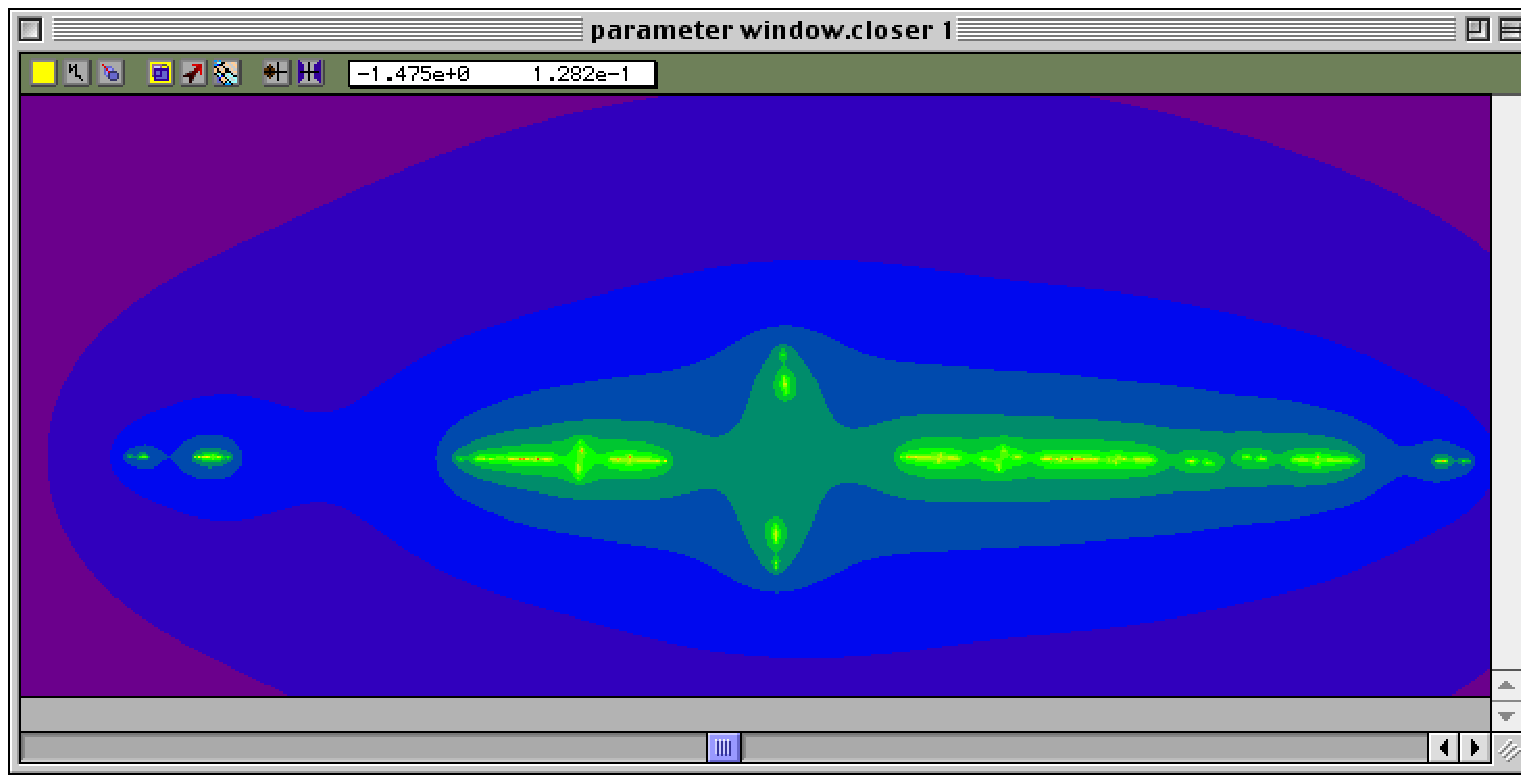}
\caption{$a\approx 2.8187+0.0119i$ and $b=0.4$.}
\label{complex1}
\end{figure}

\begin{figure}
\centering
\includegraphics[width=90mm,height=40mm]{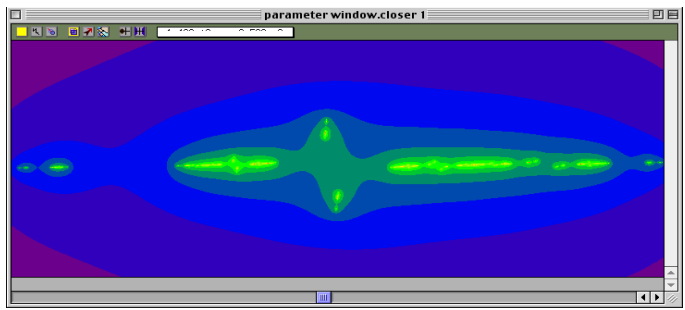}
\caption{$a\approx 2.8255-0.032i$ and $b=0.4$.}
\label{complex2}
\end{figure}

\begin{figure}
\centering
\includegraphics[width=90mm,height=40mm]{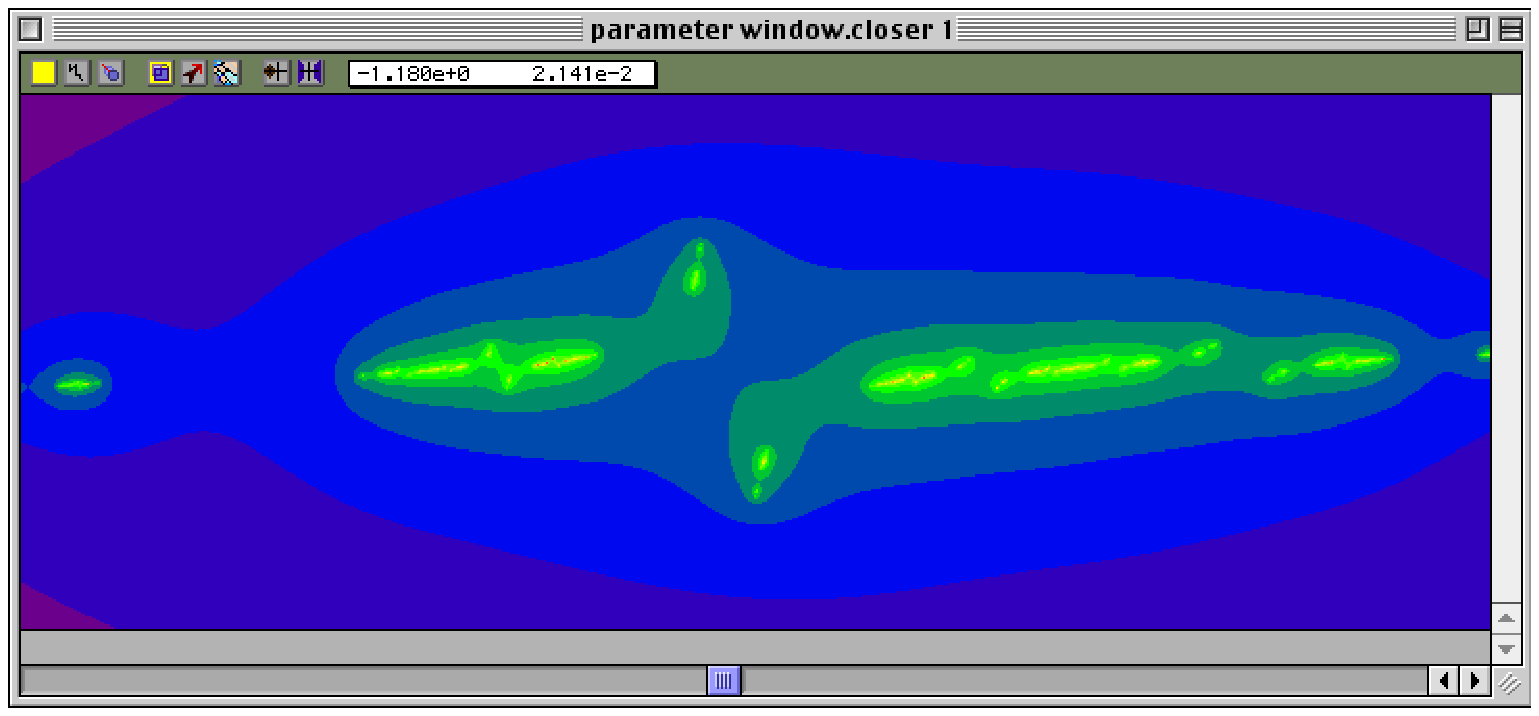}
\caption{$a\approx 2.8187-0.079i$ and $b=0.4$.}
\label{complex3}
\end{figure}

\begin{figure}
\centering
\includegraphics[width=90mm,height=40mm]{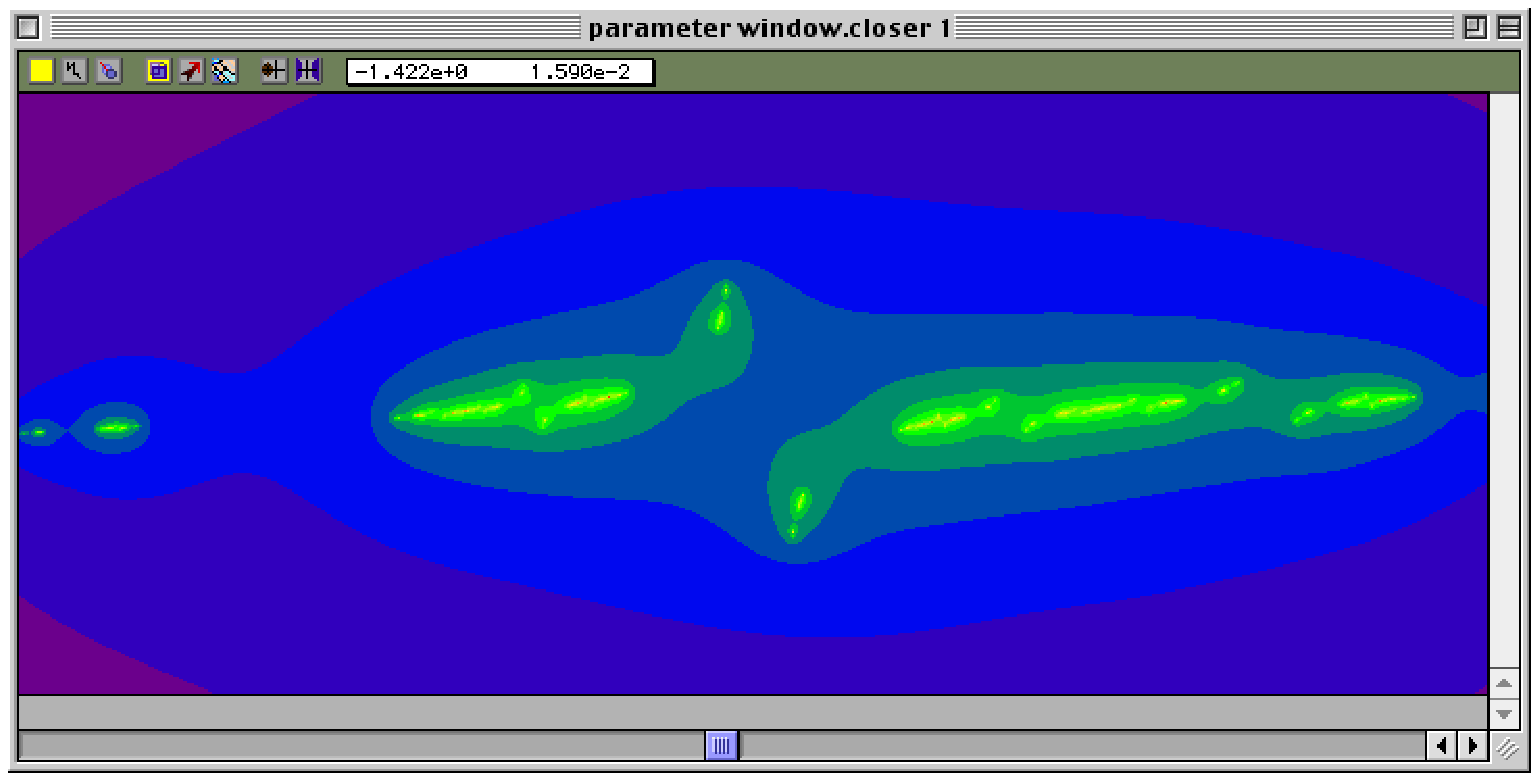}
\caption{$a\approx 2.8187-0.095i$ and $b=0.4$.}
\label{complex4}
\end{figure}

\begin{figure}
\centering
\includegraphics[width=90mm,height=40mm]{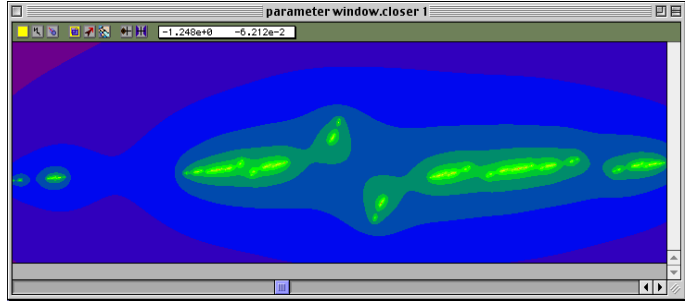}
\caption{$a\approx 2.8931-0.0959i$ and $b=0.4$.}
\label{complex5}
\end{figure}

\begin{figure}
\centering
\includegraphics[width=90mm,height=40mm]{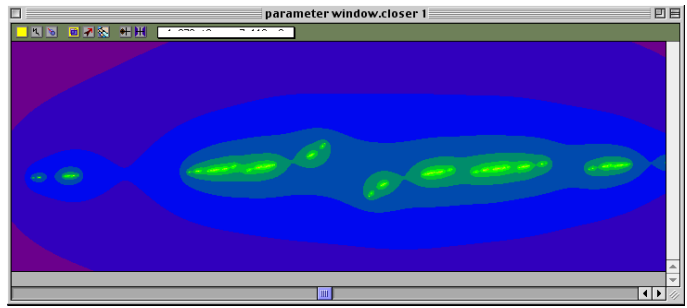}
\caption{ $a\approx 3.00-0.0919i$ and $b=0.4$.}
\label{complex6}
\end{figure}

\begin{figure}
\centering
\includegraphics[width=90mm,height=40mm]{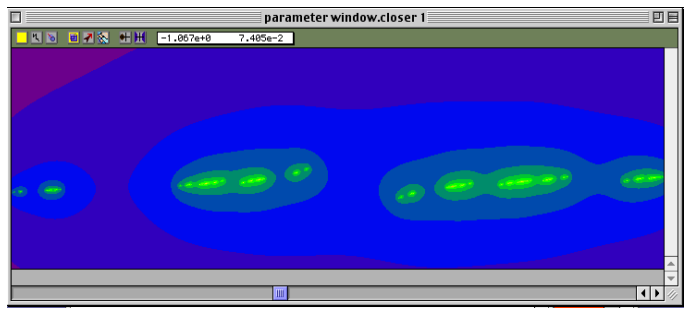}
\caption{$a\approx 3.1432-0.0919i$ and $b=0.4$}
\label{complex7}
\end{figure}

\begin{figure}
\centering
\includegraphics[width=90mm,height=40mm]{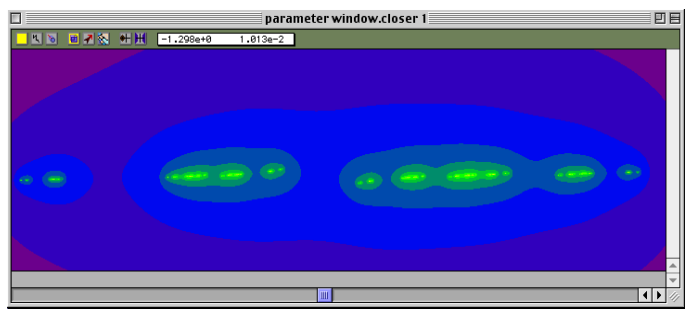}
\caption{$a\approx 3.1432-0.048i$ and $b=0.4$.}
\label{complex8}
\end{figure}

\begin{figure}
\centering
\includegraphics[width=90mm,height=40mm]{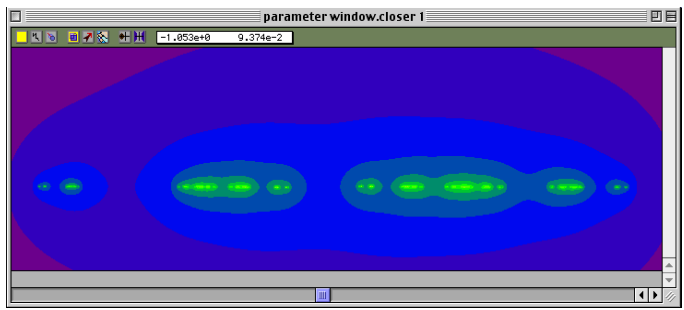}
\caption{$a\approx 3.149-0.00399i$ and $b=0.4$.}
\label{complex9}
\end{figure}
\section{Acknowledgement}

I am grateful to A. de Carvalho for showing me how to prune and
for many suggestions to improve this work.  I would like to thank 
FAPESP for their support, and Peter Hazard for helping me prepare this manuscript. 
I acknowledg the hospitality of the Instituto de Matem\'atica e Estat\'istica IME-USP 
where this work was carried out.


\begin{thebibliography}{}
%
%
\bibitem{Arai1}Arai, Z.:
{\em On hyperbolic plateaus of the H\'enon map}. Experiment. Math.
  \textbf{16(2)}, 181--188 (2007).

\bibitem{Arai2}Arai, Z.:
{\em On loops in the hyperbolic locus of the complex H\'enon maps
 and their monodromies}. ArXiv:math.DS/07042978v1, (2007).

\bibitem{BedSmi1}Bedford, E. and Smillie, J.:
{\em Real polynomial diffeomorphisms with maximal entropy: Tangencies}. Ann. of Math (2).
  \textbf{160(1)}, 1--26 (2004).

\bibitem{BedSmi2}Bedford, E. and Smillie, J.:
{\em The H\'enon family: the complex horseshoe locus and real parameter space}. 
In \textit{Complex dynamics}, volume 396 of \textit{Contemp. Math.}, 21--36.
Amer. Math. Soc., Providence, RI, (2006).

\bibitem{Cvi1}Cvitanovi\'c, P.:
{\em Periodic orbits as the skeleton of classical and quantum chaos}. Phys. D,
  \textbf{51(1-3)}, 138--151 (1991).

\bibitem{CviGunPro1}Cvitanov\'c, P. and Gunaratne, G. and Procaccia, I.:
{\em Topological and metric properties of H\'enon-type strange attractors}. Phys. Rev. A (3),
  \textbf{38(3)}, 1503--1520 (1998).

\bibitem{DavMacSan1}Davis, M.J. and Mackay, R.S. and Sannami, A.:
{\em Markov shifts in the H\'enon family}. Phys. D,
  \textbf{52(2-3)}, 171--178 (1991).

\bibitem{dCar1}de Carvalho, A.:
{\em Pruning fronts and the formation of horseshoes}. Ergodic Theory and Dynam. Systems.
  \textbf{19(4)}, 851--894 (1999).

\bibitem{DevNit1}Devaney, R. and Nitecki, Z.:
{\em Shift automorphisms in the H\'enon mapping}. Comm. Math. Phys.,
  \textbf{67(2)}, 137--146 (1979).

\bibitem{Hen1}H\'enon, M.:
{\em A two dimensional mapping with a strange attractor}. Comm. Math. Phys.
  \textbf{50(1)}, 67--77 (1976).

\bibitem{HubObe1}Hubbard, J. and Oberste-Vorth, R.:
{\em H\'enon mappings in the complex domain. I. The global topology of dynamical systems}. Inst. Hautes Études
Sci. Publ. Math.  \textbf{79}, 5--46 (1994).

\bibitem{HubObe2}Hubbard, J. and Oberste-Vorth, R.:
{\em H\'enon mappings in the complex domain. II. Projective and inductive limits of polynomials}. 
In \textit{Real and complex dynamical systems(Hilerod, 1993)}, volume 464 of 
\textit{NATO Adv. Sci. Inst. Ser. C Math. Phys. Sci.}, 89--132. Kluewer Acad. Publ., Dordrecht (1995).

\bibitem{Ishii}Ishii, Y.:
{\em Towards a kneading theory for Lozi mappings. I. A solution of the pruning
front conjecture and the first tangency problem}. Nonlinearity.  \textbf{10(3)}, 731--747 (1997).

\bibitem{Obe1}Oberste-Vorth, R.:
{\em Complex horseshoes and the dynamics of mappings of two complex variables}. PhD thesis.
 Cornell University, (1987).

\end{thebibliography}

\end{document}